\documentclass[preprint,12pt]{elsarticle}
\usepackage{eurosym}
\usepackage{lmodern}
\usepackage[utf8]{inputenc}
\usepackage[T1]{fontenc}
\usepackage[english]{babel}
\usepackage{amsfonts}
\usepackage{amsmath}
\usepackage{amssymb}

\biboptions{sort&compress}
\newtheorem{theorem}{Theorem}

\newtheorem{corollary}[theorem]{Corollary}

\newtheorem{lemma}[theorem]{Lemma}

\newtheorem{proposition}[theorem]{Proposition}

\newenvironment{proof}[1][Proof]{\noindent\textbf{#1.} }{\ \rule{0.5em}{0.5em}}
\author[affil1]{Donatella Bongiorno}
\author[affil2]{Lucian Coroianu \corref{cor1}}
\address[affil1]{ Dipartimento di Ingegneria, Universit\`{a} degli Studi di Palermo, \\
Viale delle Scienze, 90100 Palermo, Italy}
\address[affil2]{Department of Mathematics and Computer Science, University of
Oradea, \\
Universitatii 1, 410610, Oradea, Romania}
\cortext[cor1]{Corresponding author; lcoroianu@uoradea.ro}

\begin{document}

\begin{frontmatter}

\title{Uniform convergence for sequences of best $L^{p}$  approximation}

\begin{abstract}
Let $f$ be a continuous monotone real function defined on a compact interval 
$[a,b]$ of the real line. Given a sequence of partitions of $[a,b]$, $%
\Delta_n $, $\left\Vert {\Delta }_{n}\right\Vert \rightarrow 0$, and given $l\geq 0,m\geq 1$, let $\mathbf{S}_{m}^{l}(\Delta _{n})
$ be the space of all functions with the same monotonicity of $f$ that are $%
\Delta_n$-piecewise polynomial of order $m$ and that belong to the
smoothness class $C^{l}[a,b]$. \newline
In this paper we show that, for any $m\geq 2l+1$, \newline
\ \ \ $\bullet$ sequences of best $L^p$-approximation in $\mathbf{S}%
_{m}^{l}(\Delta _{n})$ converge uniformly to $f$ on any compact subinterval
of $(a,b)$;\newline
\ \ \ $\bullet$ sequences of best $L^p$-approximation in $\mathbf{S}%
_{m}^{0}(\Delta _{n})$ converge uniformly to $f$ on the whole interval $[a,b]
$.

\end{abstract}

\end{frontmatter}

\section{Introduction}

It is well known that the $L^{p}$-convergence of a sequence $(f_{n})$ to a
continuous function $f\colon \lbrack a,b]\rightarrow {\mathbb{R}}$ doesn't
imply the uniform convergence. However J.T. Lewis and O. Shisha, in \cite%
{Shisha}, proved that in the case of increasing sequences of functions $%
(f_{n})$, the $L^{p}$-convergence of $(f_{n})$ to a continuous function $f$
implies the uniform convergence on any compact interval strictly included in 
$(a,b)$. The problem of the uniform convergence of $(f_{n})$ on the whole
interval $[a,b]$ was proved by E. Kimchi and N. Richter-Dyn in \cite{Kimchi}
by considering sequences that have first-order divided differences uniformly
bounded from above and from below.

In this paper we study sequences $(f_{n})$ of piecewise polynomial functions
obtained from the $L^{p}$ minimization under monotonicity and smoothness
constraints. More precisely, we will investigate the local and the global
uniform convergence of $(f_{n})$.

Our result comes also from an idea of E. Passow \cite{Passow}.

Given a nondecreasing and continuous function $f\colon \lbrack
a,b]\rightarrow {\mathbb{R}}$, a real $p\in \lbrack 1,\infty )$, a partition 
$\Delta $ of $[a,b]$ and integers $m,l$ with $m\geq 1$ and $0\leq l\leq m-1$%
, let us consider the projection, with respect to the $L^{p}$ norm, of $f$
onto the class of all nondecreasing $\Delta $-piecewise polynomial functions
with degree less than or equal to $m$ that belongs to the smoothness class $%
C^{l}[a,b]$. Such a projection, called the best $L^{p}$ approximation of $f$
with respect to the partition $\Delta $, there exists and it is unique
(Theorem 1). Then, given a sequence $(\Delta _{n})_{n\geq 1}$ of partitions
of $[a,b]$ whose norm tends to zero, by \cite[Theorem 1]{Passow} it is
possible to generate a sequence of $(\Delta _{n})$-piecewise polynomial
functions that converges to $f$ in the $L^{p}$ norm whenever $m\geq 2l+1$.
Combining this results with the one of Lewis and Shisha mentioned earlier,
we will prove that the sequence of best $L^{p}$ approximations of $f$ along
the partitions' sequence $(\Delta _{n})_{n\geq 1}$ converges uniformly to $f$
on any compact interval included in $\left( a,b\right) $, whenever $m\geq
2l+1$. Finally, in the special case $l=0$ and $m$ arbitrary, we extend this
results to the uniform convergence of such a sequence on the whole interval $%
\left[ a,b\right] $ (Main Theorem). It is obvious that the most stringent
question is whether such result also holds for other values of $l$ if not in
general (that is, higher order of smoothness), as well as for the cases when
the condition $m\geq 2l+1$ could be dropped or replaced by a weaker one.
Clearly, this is an important open question that worth further
investigation. The proof of the Main Theorem which is long and uses many
auxiliary results (among them, the main results in \cite{Shisha}, \cite%
{Passow} and \cite{Pence}), however, it seems to work only for the case $l=0$%
.

This paper is organized as follows: in Section 2 we present some basic
notions about piecewise polynomial functions and about $L^{p}$-type
approximations for some given real $p\in \lbrack 1,\infty )$. In the same
section, we present, also, some properties of the projection operator onto
the set of monotone piecewise polynomial functions from the class $%
C^{l}[a,b] $ of degree less than or equal to $m\geq 1$ on any compact
subinterval of a given partition. In Section 3 we give some useful results
on the uniform convergence of the sequence of best $L^{p}$ approximation of $%
f$ and we enunciate the Main Theorem. In Section 4 we prove the Main
Theorem. Finally, in Section 5 some open questions, with the purpose to
complete and generalize the result of this paper, are presented.

In conclusion, remark that, since piecewise polynomial functions are used as
tools in approximation or interpolation problems, data analysis and many
others (see \cite{Boor}, \cite{Greiner}, \cite{Junbin}), $\cdots$, for
example), this note could be interesting also from the point of view of
approximation theory even if the sequences of piecewise polynomial
functions, used in the paper, cannot be expressed analytically.

\section{Piecewise polynomial approximation of $L^{p}$- integrable functions}

For $p\in [1,+\infty)$ and for a compact interval $\left[ a,b\right] $, let $%
L^{p}[a,b]$ be the space of all $L^{p}$-integrable functions on $[a,b]$,
with the norm $\left\Vert \cdot \right\Vert _{p}:L^{p}[a,b]\rightarrow
\lbrack 0,\infty ),$ 
\begin{equation*}
\left\Vert f\right\Vert _{p}=\left( \int\limits_{a}^{b}\left\vert
f(x)\right\vert ^{p}dx\right) ^{1/p}.
\end{equation*}%
It is well known that $\left( L^{p}[a,b],\left\Vert \cdot \right\Vert
_{p}\right) $ is a uniformly convex Banach space for any $p\in (1, +\infty)$.

It is also well known that, for each $p\in \lbrack 1,+\infty )$, $L^{p}[a,b]$
includes the class of all continuous functions on $[a,b]$ and the class of
all monotone functions on $[a,b]$. Moreover it is also well known that, for
some $j\in \mathbb{N}$, the class of all functions that are continuously
differentiable of order $j$ on $[a,b]$ is denoted by $C^{j}[a,b]$. In
particular, $C[a,b]$ denotes the class of all continuous functions defined
on $\left[ a,b\right] $, with the norm $\left\Vert \cdot \right\Vert
:C[a,b]\rightarrow \lbrack 0,\infty ),$ 
\begin{equation*}
\left\Vert f\right\Vert =\max_{x\in \lbrack a,b]}\left\vert f(x)\right\vert 
\text{.}
\end{equation*}

Let $\mathcal{D}[a,b]$ be the set of all partitions of $[a,b]$. If $a=\alpha
_{1}<\alpha _{2}<...<\alpha _{n}=b$ is a partition of $[a,b],$ then we
identify it with the vector ${\Delta }_{n}=\left( \alpha _{1},\alpha
_{2},...,\alpha _{n}\right) $. As usual, we denote by $\left\Vert {\Delta }%
_{n}\right\Vert $ the norm of this partition, where $\left\Vert {\Delta }%
_{n}\right\Vert =\max \{\alpha _{i+1}-\alpha _{i}:1\leq i\leq n-1\}$. A
function $f:[a,b]\rightarrow \mathbb{R}$ is called ${\Delta }_{n}$-piecewise
polynomial of order $m$, with $m\geq 1$, if the restriction of $f$ on any
interval $[\alpha _{i},\alpha _{i+1}]$, $i=\overline{1,n-1}$ is a polynomial
function of degree less than or equal to $m$. Let $\mathbf{P}_{m,{\Delta }%
_{n}}$ be the space of all functions that are ${\Delta }_{n}$-piecewise
polynomial of order $m$ and let $\mathbf{S}_{m}^{l}(\Delta _{n})=\mathbf{P}%
_{m,{\Delta }_{n}}\cap C^{l}[a,b]$ be the class of all ${\Delta }_{n}$%
-piecewise polynomials of order $m$ which are continuously differentiable up
to order $l$, for some nonnegative integer $l\leq m-1$. Moreover let $%
\mathbf{SI}_{m}^{l}(\Delta _{n})$ and $\mathbf{SD}_{m}^{l}(\Delta _{n})$ be
the set of all nondecreasing and nonincreasing functions in $\mathbf{S}%
_{m}^{l}(\Delta _{n})$, respectively. Clearly, $\mathbf{SD}_{m}^{l}(\Delta
_{n})=-\mathbf{SI}_{m}^{l}(\Delta _{n})$. In all that follows in this paper
we will consider only the case of best $L^{p}$ approximation in $\mathbf{SI}%
_{m}^{l}(\Delta _{n})$ and in the main results we will consider only
nondecreasing functions to be approximated. Obviously the case of
approximation of nondecreasing functions in $\mathbf{SD}_{m}^{l}(\Delta
_{n}) $ has the same treatment and therefore we omit the details that the
reader can easily deduce from the other case.

It is easy to observe that $\mathbf{S}_{m}^{l}(\Delta _{n})$ is a finite
dimensional linear subspace of $L^{p}[a,b]$ and that $\mathbf{SI}%
_{m}^{l}(\Delta _{n})$ and $\mathbf{SD}_{m}^{l}(\Delta _{n})$ are nonempty
closed cones in $\mathbf{S}_{m}^{l}(\Delta _{n})$. Therefore, by the
existence and uniqueness of the projection onto a closed convex subset of a
Banach uniformly convex space, it follows the existence and uniqueness of
the best approximation with respect to $\left\Vert \cdot \right\Vert _{p}$,
with $p\in (1,+\infty )$, in the spaces $\mathbf{SI}_{m}^{l}(\Delta _{n})$
and $\mathbf{SD}_{m}^{l}(\Delta _{n}),$ respectively. In additions, even if $%
L^{1}[a,b]$ is no longer a Banach uniformly convex space (nor even
reflexive), the best approximation with respect to $\left\Vert \cdot
\right\Vert _{1}$ always exists and in the particular case $f\in C[a,b]$,
the projection of $f$ onto $\mathbf{SI}_{m}^{l}(\Delta _{n})$ and $\mathbf{SD%
}_{m}^{l}(\Delta _{n})$ respectively, not only there exists but also it is
unique (see Theorem 2.3 in \cite{Pence} and see also \cite{lorentz}). Note
that in \cite{Pence} the results is obtained in a more general setting for
piecewise polynomial functions satisfying generalized convexity constraints.
Consequently, we can state the following result:

\begin{theorem}
\label{exist+uniq} If $p\in \lbrack 1,\infty )$, $m\geq 1$, $0\leq l\leq m-1$%
, $f\in C[a,b]$ and ${\Delta }_{n}\in \mathcal{D}[a,b]$, then there exists
uniquely $\mathbf{SI}_{m}^{l}(\Delta _{n},f,p)\in \mathbf{SI}_{m}^{l}(\Delta
_{n})$ so that 
\begin{equation*}
\left\Vert f-\mathbf{SI}_{m}^{l}(\Delta _{n},f,p)\right\Vert _{p}=\min_{g\in 
\mathbf{SI}_{m}^{l}(\Delta _{n})}\left\Vert f-g\right\Vert _{p}\text{.}
\end{equation*}
\end{theorem}

Taking into account the facts from above, for some given real $p\in \lbrack
1,+\infty )$, we consider the approximation operator $\mathbf{U}_{m,\Delta
_{n}}^{l}:C[a,b]\rightarrow \mathbf{SI}_{m}^{l}(\Delta _{n})$ defined by $%
\mathbf{U}_{m,\Delta _{n}}^{l}(f)=\mathbf{SI}_{m}^{l}(\Delta _{n},f,p)$.

The following properties are immediate and therefore we omit their proofs.

\begin{proposition}
\label{properties approx} Let $f\in C[a,b]$. For some $p\in \lbrack 1,
+\infty)$, $m\geq 1$, $0\leq l\leq m-1$ and ${\Delta }_{n}\in \mathcal{D}%
[a,b]$ with ${\Delta }_{n}=(\alpha _{1},\alpha _{2},...,\alpha _{n})$, we
have that:

(i) $\mathbf{SI}_{m}^{l}(\Delta _{n},f-c,p)=\mathbf{SI}_{m}^{l}(\Delta
_{n},f,p)-c$, for all $c\in \mathbb{R};$

(ii) $\mathbf{SI}_{m}^{l}(\Delta _{n},cf,p)=c\cdot \mathbf{SI}%
_{m}^{l}(\Delta _{n},f,p)$ and $\mathbf{SD}_{m}^{l}(\Delta _{n},cf,p)=c\cdot 
\mathbf{SD}_{m}^{l}(\Delta _{n},f,p)$, for all $c\in \left[ 0,\infty \right]
;$
\end{proposition}

In all that follows, we no longer impose that ${\Delta }_{n}\in \mathcal{D}%
[a,b]$ has exactly $n$ knots. Indeed, later on we construct sequences of
partitions of $[a,b]$ such that any partition of the sequence has an
arbitrary number of knots. Therefore, in general, we will consider a vector
of the form ${\Delta }_{n}=(\alpha _{1},\alpha _{2},...,\alpha _{k_{n}})$.

In the next result we will provide sufficient conditions which will imply
that the sequence $\left( \mathbf{SI}_{m}^{l}(\Delta _{n},f,p)\right)
_{n\geq 1}$ converges to $f$ in the $L^{p}$ norm $\left\Vert \cdot
\right\Vert _{p}$ for some given $p\in \lbrack 1,\infty )$. Combining this
result with several auxiliary results, in the next section we will be able
to prove a local uniform convergence for this sequence and even a global
uniform convergence in the special case $l=0$. At the moment, the next
result will not suffice to obtain these uniform convergence properties since
we will rely on a rate of convergence for an estimation in the $L^{p}$ norm
(see relation (1) in the proof) and it is well known that in general such
estimations will not necessarily imply convergence in the uniform norm based
on the simple fact that these two norms are not equivalent.

\begin{proposition}
\label{convergence L2}Suppose that $f:[a,b]\rightarrow \mathbb{R}$ is a
continuous function and consider in $\mathcal{D}[a,b]$ a sequence $\left({\
\Delta }_{n}\right) _{n\geq 1}$ with $\left\Vert {\Delta }_{n}\right\Vert
\rightarrow 0$. Moreover, let $l$ be an arbitrary nonnegative integer. Then
if $f$ is nondecreasing, for any $p\in \lbrack 1,\infty )$ and $m\geq 2l+1$
we have:

\begin{equation*}
\left\Vert f-\mathbf{SI}_{m}^{l}(\Delta _{n},f,p)\right\Vert _{p}\leq
(b-a)^{1/p}\omega _{f}(\left\Vert {\Delta }_{n}\right\Vert )\text{, for all }%
n\geq 1\text{,}
\end{equation*}
where

\begin{equation*}
\omega _{f}(\delta )=\sup \{\left\vert f(x)-f(y)\right\vert :x,y\in \lbrack
a,b]\text{,}\left\vert x-y\right\vert \leq \delta \}\text{}
\end{equation*}
denotes the modulus of uniform continuity of $f.$

So, in particular, we have: 
\begin{equation*}
\lim_{n\rightarrow \infty }\left\Vert f-\mathbf{SI}_{m}^{l}(\Delta
_{n},f,p)\right\Vert _{p}=0.
\end{equation*}
\end{proposition}

\begin{proof}
Let us choose arbitrary $n\geq 1$ and suppose that ${\Delta }_{n}=(\alpha
_{1},\alpha _{2},...,\alpha _{k_{n}})$. In \cite{Passow} (see Theorem 1
there) it is proved that if $\left( x_{i},y_{i}\right) _{i=1,s}$ \ are in $%
\mathbb{R}^{2}$ with $x_{1}<x_{2}<\cdot \cdot \cdot <x_{s}$ then there
exists $P\in \mathbf{S}_{2l+1}^{l}(\Delta _{s})$, where $\Delta
_{s}=(x_{1},...,x_{s})$, such that $P(x_{i})=y_{i}$, $i=\overline{1,s}$ and $%
P$ is monotone on any interval $\left[ x_{i},x_{i+1}\right] $, $i=\overline{%
1,s-1}$. Then, applying this result to $(\alpha _{i},f(\alpha
_{i}))_{i=1,k_{n}}$ we get $P\in \mathbf{SI}_{2l+1}^{l}(\Delta _{n})$ that
interpolates $f$ on the knots of ${\Delta }_{n}$, and since $m\geq 2l+1$, it
follows that $P\in \mathbf{SI}_{m}^{l}(\Delta _{n})$.

So, since $f$ and $P$ are both nondecreasing and since $f(\alpha
_{i})=P(\alpha _{i})$, $i=\overline{1,k_{n}}$, it easily results that for
any $i\in \{1,...,k_{n}-1\}$ and $x\in \lbrack \alpha _{i},\alpha _{i+1}]$
we have $\left\vert f(x)-P(x)\right\vert \leq f(\alpha _{i+1})-f(\alpha
_{i}).$ Moreover, since 
\begin{equation*}
f(\alpha _{i+1})-f(\alpha _{i})\leq \omega _{f}(\alpha _{i+1}-\alpha
_{i})\leq \omega _{f}(\left\Vert {\Delta }_{n}\right\Vert )
\end{equation*}%
for all $i\in \{1,...,k_{n}-1\}$, we have 
\begin{equation*}
\left\vert f(x)-P(x)\right\vert \leq \omega _{f}(\left\Vert {\Delta }%
_{n}\right\Vert )\text{, for all }x\in \lbrack a,b]\text{.}
\end{equation*}%
This easily implies 
\begin{equation*}
\left\Vert f-P\right\Vert _{p}\leq (b-a)^{1/p}\omega _{f}(\left\Vert {\Delta 
}_{n}\right\Vert ).
\end{equation*}

Then, by the definition of $\mathbf{SI}_{m}^{l}(\Delta _{n},f,p)$ and since $%
P\in \mathbf{SI}_{m}^{l}(\Delta _{n})$, it results that 
\begin{equation}
\left\Vert f-\mathbf{SI}_{m}^{l}(\Delta _{n},f,p)\right\Vert _{p}\leq
\left\Vert f-P\right\Vert _{p}\leq (b-a)^{1/p}\omega _{f}(\left\Vert {\Delta 
}_{n}\right\Vert ).  \label{b}
\end{equation}%
The continuity of $f$ and the fact that $\left\Vert {\Delta }_{n}\right\Vert
\rightarrow 0$ implies $\omega _{f}(\left\Vert {\Delta }_{n}\right\Vert
)\rightarrow 0$, hence by (\ref{b}) we have $\lim_{n\rightarrow \infty
}\left\Vert f-\mathbf{SI}_{m}^{l}(\Delta _{n},f,p)\right\Vert _{p}=0$ and
the proof is complete.
\end{proof}

\section{Uniform convergence of best piecewise polynomial approximations}

This section is dedicated to the study of the uniform convergence of the
sequence of all piecewise polynomial functions $(f_{n})_{n\geq 1}$ to a
continuous nondecreasing function $f$ of the sequence $(\mathbf{SI}%
_{m}^{l}(\Delta _{n},f,p))_{n\geq 1}$ whenever $\left\Vert {\Delta }%
_{n}\right\Vert \rightarrow 0$.

It is well known that if $f$ and $(f_{n})_{n\geq 1}$ are continuous
functions on $[a,b]$ such that $(f_{n})_{n\geq 1}$ converges uniformly to $f$
on $[a,b]$ then $(f_{n})_{n\geq 1}$ converges towards $f$ in the $L^{p}$
norm for any $p\in [1, \infty)$. But it is also well known that the converse
property does not hold in general. However, if $(f_{n})_{n\geq 1}$ is a
sequence of nondecreasing functions then the following theorem provides us a
local convergence property in the uniform norm.

\begin{theorem}
\label{T4} \label{uniform_converg_incr}(see \cite[Theorem 1]{Shisha}) Let $f$
be a real, continuous function on the finite interval $(a,b)$, and let $%
(f_{n})_{n\geq 1}$ be a sequence of nondecreasing functions on $(a,b)$ such
that $\lim\limits_{n\rightarrow \infty }\left\Vert f_{n}-f\right\Vert _{p}=0$%
, where $1\leq p<\infty $. Then for any $c$ and $d$ with $a<c<d<b$, the
sequence $(f_{n})_{n\geq 1}$ converges uniformly to $f$ on $[c,d]$.
\end{theorem}

Combining Proposition \ref{convergence L2} with Theorem \ref{T4}, we obtain
the following result.

\begin{corollary}
\label{uniform_conv_local}Suppose that $f:[a,b]\rightarrow \mathbb{R}$ is a
continuous nondecreasing function and consider in $\mathcal{D}[a,b]$ a
sequence $\left( {\Delta }_{n}\right) _{n\geq 1}$, such that $\left\Vert {%
\Delta }_{n}\right\Vert \rightarrow 0$. Then for any nonnegative integer $l$
and for any $m\geq 2l+1$, the sequence $\left( \mathbf{SI}_{m}^{l}(\Delta
_{n},f,p)\right) _{n\geq 1}$ is uniformly convergent to $f$ on any interval $%
\left[ c,d\right] $, with $a<c<d<b$.
\end{corollary}

\begin{proof}
By Proposition 3, we have $\lim\limits_{n\rightarrow \infty }\left\Vert 
\mathbf{SI}_{m}^{l}(\Delta _{n},f,p)-f\right\Vert _{p}=0$. Hence, we obtain
the conclusion taking $f_{n}=\mathbf{SI}_{m}^{l}(\Delta _{n},f,p)$ for any $%
n $ in Theorem \ref{T4}.
\end{proof}

Now let us consider the function $f\equiv 0$ and the sequence $%
(f_{n})_{n\geq 1}$ defined on $[0,1]$ by $f_{n}(x)=x^{n},n=1,2,\cdots $. It
is trivial to observe that $\Vert f_{n}\Vert _{p}\rightarrow 0$, for any $%
p\in \lbrack 1,\infty )$ and that $(f_{n})_{n\geq 1}$ is uniformly
convergent to zero on any $[c,d]$, with $0<c<d<1$. However, $(f_{n})_{n\geq
1}$ is not uniformly convergent to zero on $[0,1]$.

Let us note that in the previous corollary as well as in Proposition 3 which
implies this corollary, we need to assume that $m\geq 2l+1$ because
otherwise we cannot use Theorem 1 in \cite{Passow} used to prove Proposition
3. Actually, in the case $m\leq 2l$ the interpolant used to prove
Proposition 3 may not exist (see Remark 2 in \cite{Passow}). Of course, it
is an open question whether these results would hold also for $m\leq 2l$ but
then another proof is needed which does not use Theorem 1 in \cite{Passow}.

In this paper we prove that:

\noindent \textbf{Main Theorem} \label{UC piecewise approx} \textit{If $%
f:[a,b]\rightarrow \mathbb{R}$ is a continuous nondecreasing function and if 
$\left( {\Delta }_{n}\right) _{n\geq 1}\in\mathcal{D}[a,b]$ such that $%
\left\Vert {\Delta }_{n}\right\Vert \rightarrow 0$, then, for any real $p\in
\lbrack 1,\infty )$ and $m\geq 1$, the sequence $\left( \mathbf{SI}%
_{m}^{0}(\Delta _{n},f,p)\right) _{n\geq 1}$ is uniformly convergent to $f$
on $[a,b]$; that is $\lim\limits_{n\rightarrow \infty }\left\Vert \mathbf{SI}%
_{m}^{0}(\Delta _{n},f)-f\right\Vert =0$.}

To prove this Theorem we need some auxiliary results.

\begin{lemma}
\label{uc nondecreasing}(see \cite[Lemma 4.3]{Bogachev}) Suppose that $%
f:[a,b]\rightarrow \mathbb{R}$ is a continuous monotone function and suppose
that $(f_{n})_{n\geq 1}$ is a sequence of monotone functions, all with the
same monotonicity as $f$, such that for some $p\geq 1$ we have $%
\lim\limits_{n\rightarrow \infty }\left\Vert f_{n}-f\right\Vert _{p}=0$. If $%
\lim\limits_{n\rightarrow \infty }f_{n}(a)=f(a)$ and $\lim\limits_{n%
\rightarrow \infty }f_{n}(b)=f(b)$, then $\lim\limits_{n\rightarrow \infty
}\left\Vert f_{n}-f\right\Vert =0$.
\end{lemma}

First of all we give a simple sufficient condition to obtain the pointwise
convergence needed in the previous lemma.

\begin{lemma}
\label{Pointwise_upper}Suppose that $f:[a,b]\rightarrow \mathbb{R}$ is a
nondecreasing continuous function and suppose that $(f_{n})_{n\geq 1}$ is a
sequence of nondecreasing functions, such that for some $p\geq 1$ we have $%
\lim\limits_{n\rightarrow \infty }\left\Vert f_{n}-f\right\Vert _{p}=0$. If $%
\left( x_{n}\right) _{n\geq 1}$ is a sequence in $[a,b]$ such that $%
x_{n}\rightarrow a$ and $f_{n}(x_{n})\geq f(x_{n})$ for all $n\geq 1$, then $%
\lim\limits_{n\rightarrow \infty }f_{n}(x_{n})=\lim\limits_{n\rightarrow
\infty }f(x_{n})=f(a)$. In particular, if $f_{n}(a)\geq f(a)$ for $all$ $%
n\geq 1$, then $\lim\limits_{n\rightarrow \infty }f_{n}(a)=f(a)$.
\end{lemma}

\begin{proof}
To avoid unnecessary reasonings with subsequences, we may assume that $%
\lim\limits_{n\rightarrow \infty }f_{n}(x_{n})$ exists, including the case
when it is unbounded. By way of contradiction, suppose that $%
\lim\limits_{n\rightarrow \infty }f_{n}(x_{n})=f(a)$ does not hold.\textrm{\ 
}Then, since $f(x_{n})\rightarrow f(a)$, there exists $\varepsilon >0$ such
that for sufficiently large $n$ we have $f_{n}(x_{n})-f(x_{n})>\varepsilon $%
. Letting $n\rightarrow \infty $ we get $\lim\limits_{n\rightarrow \infty
}f_{n}(x_{n})-f(a)>\varepsilon /2$. From the continuity of $f$ it results
the existence of $y_{0}\in \lbrack a,b]$ such that $\lim\limits_{n%
\rightarrow \infty }f_{n}(x_{n})-f(x)>\varepsilon /2$ for all $x\in \lbrack
a,y_{0}]$. In particular we get $\lim\limits_{n\rightarrow \infty
}f_{n}(x_{n})-f(y_{0})>\varepsilon /2$. This implies the existence of $%
n_{1}\in \mathbb{N}$ such that 
\begin{equation}
f_{n}(x_{n})-f(y_{0})>\varepsilon /2,\text{ for all }n\geq n_{1}.
\label{epsilon/2}
\end{equation}%
Let us choose arbitrary $y_{1}\in (a,y_{0})$. Since $x_{n}\rightarrow a$
there is $n_{2}\in \mathbb{N}$ (which may depends only on $y_{1}$) such that 
$x_{n}\leq y_{1}$ for all $n\geq n_{2}$. By the monotonicity of $f$ and $%
f_{n}$, for each $n$, and by (\ref{epsilon/2}), it results 
\begin{equation*}
f(x)\leq f(y_{0})<f_{n}(x_{n})\leq f_{n}(x)\text{, for all }n\geq \max
\{n_{1},n_{2}\}\text{ and }x\in \lbrack y_{1},y_{0}]\text{.}
\end{equation*}%
This implies 
\begin{eqnarray*}
&&\int\limits_{a}^{b}\left\vert f_{n}(x)-f(x)\right\vert ^{p}dx \\
&\geq &\int\limits_{y_{1}}^{y_{0}}\left\vert f_{n}(x)-f(x)\right\vert
^{p}dx\geq (y_{0}-y_{1})\left\vert f_{n}(x_{n})-f(y_{0})\right\vert
^{p}>(y_{0}-y_{1})\cdot \left( \frac{\varepsilon }{2}\right) ^{p}
\end{eqnarray*}%
for all $n\geq \max \{n_{1},n_{2}\}$.\newline
Clearly this implies that we cannot have $\lim\limits_{n\rightarrow \infty
}\left\Vert f_{n}-f\right\Vert _{p}=0$ and this is a contradiction. In
conclusion, we have $\lim\limits_{n\rightarrow \infty
}f_{n}(x_{n})=\lim\limits_{n\rightarrow \infty }f(x_{n})=f(a)$$.$
\end{proof}

The following lemma is inspired by the inequality of Markov for real
polynomials.

\begin{lemma}
\label{Marker}Let $P:[a,b]\rightarrow \mathbb{R}$ be a nondecreasing
polynomial function of degree at most $m\geq 1$. Then, 
\begin{equation*}
P\left( a+\frac{b-a}{2m^{2}+1}\right) \leq P(a)+\frac{2m^{2}}{2m^{2}+1}\cdot
(P(b)-P(a))\text{.}
\end{equation*}
\end{lemma}

\begin{proof}
We start with the special case $a=0$, $b=1$, $P(0)=0$ and $P(1)=1$. Let us
define on $[-1,1]$ the polynomial $Q(x)=P((x+1)/2) $. From the Markov's
inequality we have 
\begin{equation*}
\max_{-1\leq x\leq 1}\left\vert Q^{\prime }(x)\right\vert \leq m^{2}\cdot
\max_{-1\leq x\leq 1}\left\vert Q(x)\right\vert
\end{equation*}%
. Since $Q^{\prime }(x)=\frac{1}{2}P^{\prime }\left( \frac{x+1}{2}\right) $,
this implies 
\begin{equation*}
\max_{-1\leq x\leq 1}\left\vert P^{\prime }\left( \frac{x+1}{2}\right)
\right\vert \leq 2m^{2}\cdot \max_{-1\leq x\leq 1}\left\vert P\left( \frac{%
x+1}{2},\right) \right\vert
\end{equation*}
hence 
\begin{equation*}
\max_{0\leq x\leq 1}\left\vert P^{\prime }(x)\right\vert \leq 2m^{2}\cdot
\max_{0\leq x\leq 1}\left\vert P(x)\right\vert =2m^{2}\text{.}
\end{equation*}
Then, by the mean value theorem, we have 
\begin{equation*}
P\left( \frac{1}{2m^{2}+1}\right) =\int\limits_{0}^{1/(2m^{2}+1)}P^{\prime
}(x)dx=\frac{1}{2m^{2}+1}P^{\prime }(x_{m})\leq \frac{2m^{2}}{2m^{2}+1}.
\end{equation*}

Now, let us consider the general case. It is immediate that $%
P_{1}:[0,1]\rightarrow \mathbb{R}$, defined as $%
P_{1}(x)=(P(a+(b-a)x)-P(a))/(P(b)-P(a))$, is a nondecreasing polynomial of
degree at most $m$, such that $P_{1}(0)=0$ and $P_{1}(1)=1$. It means that $%
P_{1}\left( \frac{1}{2m^{2}+1}\right) \leq \frac{2m^{2}}{2m^{2}+1}$, from
which the desired conclusion easily follows.
\end{proof}

\section{Proof of the Main Theorem}

If $f$ is a constant, then the conclusion is immediate since $\mathbf{SI}%
_{m}^{0}(\Delta _{n},f,p)=f$ for any $n$. Therefore, in what it follows, we
assume that $f$ is not a constant. By Corollary \ref{uniform_conv_local}, it
results that $\left( \mathbf{SI}_{m}^{0}(\Delta _{n},f,p)\right) _{n\geq 1}$
converges uniformly to $f$ on any interval $[c,d]$ with $a<c<d<b$. In order
to prove the desired uniform convergence on \ $\left[ a,b\right] $, it is
enough to prove that $\lim\limits_{n\rightarrow \infty }\mathbf{SI}%
_{m}^{0}(\Delta _{n},f,p)(a)=f(a)$ and $\lim\limits_{n\rightarrow \infty }%
\mathbf{SI}_{m}^{0}(\Delta _{n},f,p)(b)=f(b)$, respectively (see, e. g.
Lemma 4.3 in \cite{Bogachev}). Without loss of generality we may assume that 
$f(a)=0$. In fact, supposed that the thesis holds for such functions, given
a continuous nondecreasing generic function $f$ we set $g(x)=f(x)-f(a).$
Then $g$ is continuous nondecreasing and $g(a)=0$. So $\lim\limits_{n%
\rightarrow \infty }\left\Vert \mathbf{SI}_{m}^{0}(\Delta
_{n},g,p)-g\right\Vert =0$. Now, by Proposition \ref{properties approx}(i),
we have $\mathbf{SI}_{m}^{0}(\Delta _{n},g,p)$ $=\mathbf{SI}_{m}^{0}(\Delta
_{n},f,p)-f(a)$, for all $n\geq 1$. Therefore 
\begin{eqnarray*}
\lefteqn{\lim\limits_{n\rightarrow \infty }\left\Vert \mathbf{SI}%
_{m}^{0}(\Delta _{n},f,p)-f\right\Vert } \\
&=&\lim\limits_{n\rightarrow \infty }\left\Vert \mathbf{SI}_{m}^{0}(\Delta
_{n},f,p)-f(a)+f(a)-f\right\Vert \\
&=&\lim\limits_{n\rightarrow \infty }\left\Vert \mathbf{SI}_{m}^{0}(\Delta
_{n},g,p)-g\right\Vert \\
&=&0.
\end{eqnarray*}

From now on, until the end of the proof, we will set $f_{n}=\mathbf{SI}%
_{m}^{0}(\Delta _{n},f,p)$ for all $n\geq 1$. To avoid the use of
subsequences we may suppose that there are only two cases: either $%
f_{n}(a)<f(a)=0$ for all $n\geq 1$, or $f_{n}(a)\geq f(a)=0$ for all $n\geq
1 $. But in this latter one, by Lemma \ref{Pointwise_upper}, it results that 
$\lim\limits_{n\rightarrow \infty }f_{n}(a)=f(a)$ and therefore in this case
there is nothing to be proved. So, in all that follows we will suppose that $%
f_{n}(a)<f(a)=0$ for all $n\geq 1$.

The idea of the proof is to show that any subsequence of $\left(
f_{n}(a)\right) _{n\geq 1}$ contains a subsequence that converges to $0$.
Clearly, this will imply that $\lim\limits_{n\rightarrow \infty }f_{n}(a)=0$%
. For this reason, without loosing generality we may suppose that there
exists the limit $\lim\limits_{n\rightarrow \infty }f_{n}(a)$ (including the
case when this limit is $-\infty $). For any $n\geq 1$ let ${\Delta}
_{n}=(\alpha _{1}(n),\alpha _{2}(n),...,\alpha _{k_{n}}(n))$. We must
observe that for each $n$ the function $f_{n}$ cannot be a constant
function. Indeed, if $f_{n}$ would be constant (with strictly negative
constant value according to our assumption) then, taking $h(x)=0$, $x\in
\lbrack a,b]$, we obviously have $h\in \mathbf{SI}_{m}^{0}(\Delta _{n})$ and
in addition one can easily prove that $\left\Vert h-f\right\Vert
_{p}<\left\Vert f_{n}-f\right\Vert _{p}$, which contradicts the fact that $%
f_{n}$ is the best approximation of $f$ in $\mathbf{SI}_{m}^{0}(\Delta _{n})$
with respect to $\left\Vert \cdot \right\Vert _{p}$. Therefore, since $f_{n}$
is not constant, it results that there exists $j\in \{1,...,k_{n}-1\}$ such
that $f_{n}(\alpha _{j}(n))=f_{n}(a)<f_{n}(\alpha _{j+1}(n))$. In addition,
we notice that necessarily $f_{n}$ is strictly increasing on $\left[ \alpha
_{j}(n),\alpha _{j+1}(n)\right] $ (otherwise $f_{n}$ is not polynomial on
this interval). Since $f$ is not a constant function let $a_{1}\in \lbrack
a,b]$ such that $f(a_{1})=f(a)=0<f(x)$, for all $x\in (a_{1},b]$. Without
any loss of generality, suppose that $\left( \alpha _{j}(n)\right) _{n\geq
1} $ is convergent and denote its limit with $u_{0}$. If $u_{0}>a_{1}$ then
for some fixed value $x_{0}\in (a_{1},u_{0})$ and sufficiently large $n$ we
have $f_{n}(x_{0})\leq f_{n}(\alpha _{j}(n))<0<f(x_{0})$. It means that $%
f_{n}(x_{0})$ does not converge to $f(x_{0}).$ This clearly contradicts
Corollary \ref{uniform_conv_local}. Therefore we have $u_{0}\leq a_{1}$. In
particular, we have that $f(u_{0})=f(a)=0$. Now, let us prove that $%
f_{n}(\alpha _{j+1}(n))>f(a)=0$ for all $n\geq 1$. By way of contradiction,
let us consider, for some $n\geq 1,$ the function $h_{n}:[a,b]\rightarrow 
\mathbb{R}$, $h_{n}(x)=f_{n}(\alpha _{j+1}(n))$ if $x\in \lbrack
a,a_{j+1}(n)]$ and $h_{n}(x)=f_{n}(x)$ elsewhere. Then $h_{n}\in \mathbf{SI}%
_{m}^{0}(\Delta _{n})$ and $f_{n}(x)<h_{n}(x)\leq f(x)$, for all $x\in
\lbrack a,\alpha _{j+1}(n))$ (here it is important that $f_{n}$ is strictly
increasing on $\left[ \alpha _{j}(n),\alpha _{j+1}(n)\right] $, that $%
f_{n}(\alpha _{j+1}(n))\leq 0$ as well as the monotonicity of $f$).
Consequently, this implies that 
\begin{equation*}
\int\limits_{a}^{\alpha _{j+1}(n)}\left\vert f(x)-h_{n}(x)\right\vert
^{p}dx<\int\limits_{a}^{\alpha _{j+1}(n)}\left\vert f(x)-f_{n}(x)\right\vert
^{p}dx
\end{equation*}%
that easily implies 
\begin{equation*}
\int\limits_{a}^{b}\left\vert f(x)-h_{n}(x)\right\vert
^{p}dx<\int\limits_{a}^{b}\left\vert f(x)-f_{n}(x)\right\vert ^{p}dx.
\end{equation*}%
Thus, $\left\Vert h_{n}-f\right\Vert _{p}<\left\Vert f_{n}-f\right\Vert _{p}$
which, since $h_{n}\in \mathbf{SI}_{m}^{0}(\Delta _{n})$, contradicts the
fact that $f_{n}$ is the best approximation of $f$ in $\mathbf{SI}%
_{m}^{0}(\Delta _{n})$ with respect to $\left\Vert \cdot \right\Vert _{p}$.
Therefore, we have $f_{n}(\alpha _{j+1}(n))>f(a)=0$ for all $n\geq 1$.

Now let us show that $\lim\limits_{n\rightarrow \infty }f_{n}(\alpha
_{j+1}(n))=0$. \ By way of contradiction we suppose that this is not true.
In this case we may assume that there exists $\gamma _{1}>0$ such that $%
f_{n}(\alpha _{j+1}(n))>\gamma _{1}$ for sufficiently large $n$. Since $%
f(u_{0})=0$ and by the continuity of $f$, there exists $\delta >u_{0}$ such
that $f(x)<\gamma _{1}/2$ for all $x\in \lbrack u_{0},\delta ]$. Since $%
\lim\limits_{n\rightarrow \infty }\alpha _{j+1}(n)=u_{0}$, for sufficiently
large $n$, we have $\alpha _{j+1}(n)<\left( u_{0}+\delta \right) /2$.
Therefore, by the monotonicity of $f_{n}$, for sufficiently large $n$ we get 
\begin{equation*}
f(x)<\gamma _{1}/2<\gamma _{1}<f_{n}(\alpha _{j+1}(n))\leq f_{n}(x)\text{,
for all }x\in \lbrack \left( u_{0}+\delta \right) /2,\delta ].
\end{equation*}%
This implies that, for some $x_{0}\in \lbrack \left( u_{0}+\delta \right)
/2,\delta ]$, $f_{n}(x_{0})$ does not converge to $f(x_{0})$; a
contradiction. Therefore, we have $\lim\limits_{n\rightarrow \infty
}f_{n}(\alpha _{j+1}(n))=0$.

Now, suppose that $n\geq 1$ is fixed and, to simplify the notations, let us
denote $v_{n}=f_{n}(\alpha _{j+1}(n))$. Let $g_{n}: \left[ a,b\right]
\rightarrow \mathbb{R}$, defined by 
\begin{equation*}
g_{n}(x)=\left\{ 
\begin{array}{ccc}
0, & if & x\in \lbrack a,\alpha _{j}(n)]; \\ 
\frac{\left( x-\alpha _{j}(n)\right) v_{n}}{\alpha _{j+1}(n)-\alpha _{j}(n)},
& if & x\in \left[ \alpha _{j}(n),\alpha _{j+1}(n)\right]; \\ 
f_{n}(x), & if & x\in \left[ \alpha _{j+1}(n),b\right] .%
\end{array}
\right.
\end{equation*}
It is trivial to observe that $g_{n}\in \mathbf{SI}_{m}^{0}(\Delta _{n})$,
so 
\begin{equation*}
\left\Vert f_{n}-f\right\Vert _{p}\leq \left\Vert g_{n}-f\right\Vert _{p}%
\text{.}
\end{equation*}
We also easily notice that 
\begin{equation*}
\int\limits_{a}^{\alpha _{j}(n)}\left\vert f(x)-f_{n}(x)\right\vert
^{p}dx>\int\limits_{a}^{\alpha _{j}(n)}\left\vert f(x)-g_{n}(x)\right\vert
^{p}dx
\end{equation*}%
and%
\begin{equation*}
\int\limits_{\alpha _{j+1}(n)}^{b}\left\vert f(x)-f_{n}(x)\right\vert
^{p}dx=\int\limits_{\alpha _{j+1}(n)}^{b}\left\vert f(x)-g_{n}(x)\right\vert
^{p}dx\text{.}
\end{equation*}%
Hence%
\begin{equation*}
\int\limits_{\alpha _{j}(n)}^{\alpha _{j+1}(n)}\left\vert
f(x)-f_{n}(x)\right\vert ^{p}dx<\int\limits_{\alpha _{j}(n)}^{\alpha
_{j+1}(n)}\left\vert f(x)-g_{n}(x)\right\vert ^{p}dx\text{.}
\end{equation*}%
This further implies that 
\begin{equation}  \label{dise}
\int\limits_{\alpha _{j}(n)}^{u_{n}}\left\vert f(x)-f_{n}(x)\right\vert
^{p}dx<\int\limits_{\alpha _{j}(n)}^{\alpha _{j+1}(n)}\left\vert
f(x)-g_{n}(x)\right\vert ^{p}dx,
\end{equation}%
where $u_{n}=\alpha _{j}(n)+\frac{1}{2m^{2}+1}\left( \alpha _{j+1}(n)-\alpha
_{j}(n)\right) $.

Applying the mean value theorem in both integrals in (\ref{dise}), there
exist $c_{n}\in \left( \alpha _{j}(n),u_{n}\right) $ and $d_{n}\in \left(
\alpha _{j}(n),\alpha _{j+1}(n)\right) $, such that%
\begin{equation*}
\frac{\alpha _{j+1}(n)-\alpha _{j}(n)}{2m^{2}+1}\cdot \left\vert
f(c_{n})-f_{n}(c_{n})\right\vert ^{p}<\left( \alpha _{j+1}(n)-\alpha
_{j}(n)\right) \left\vert f(d_{n})-g_{n}(d_{n})\right\vert ^{p}\text{,}
\end{equation*}%
which by simple calculations gives%
\begin{equation}
\left\vert f(c_{n})-f_{n}(c_{n})\right\vert <\left( 2m^{2}+1\right)
^{1/p}\left\vert f(d_{n})-g_{n}(d_{n})\right\vert \text{.}
\label{partial ineq}
\end{equation}%
As $f_{n}$ is nondecreasing, by Lemma \ref{Marker}, we get that 
\begin{eqnarray*}
f_{n}\left( \alpha _{j}(n)\right) &\leq &f_{n}(c_{n})\leq f_{n}\left(
u_{n}\right) \\
&\leq &f_{n}\left( \alpha _{j}(n)\right) +\frac{2m^{2}}{2m^{2}+1}\cdot
(f_{n}\left( \alpha _{j+1}(n)\right) -f_{n}\left( \alpha _{j}(n)\right) )%
\text{.}
\end{eqnarray*}%
This implies that there exists $t_{n}\in \left[ 0,\frac{2m^{2}}{2m^{2}+1}%
\right] $, such that 
\begin{eqnarray*}
f_{n}(c_{n}) &=&f_{n}\left( \alpha _{j}(n)\right) +t_{n}\cdot (f_{n}\left(
\alpha _{j+1}(n)\right) -f_{n}\left( \alpha _{j}(n)\right) ) \\
&=&(1-t_{n})\cdot f_{n}\left( \alpha _{j}(n)\right) +t_{n}\cdot f_{n}\left(
\alpha _{j+1}(n)\right) \\
&=&(1-t_{n})\cdot f_{n}\left( a\right) +t_{n}\cdot v_{n}\text{.}
\end{eqnarray*}%
Returning to inequality (\ref{partial ineq}), we get 
\begin{equation}  \label{ineq2}
\left\vert (1-t_{n})\cdot f_{n}\left( a\right) +t_{n}\cdot
v_{n}-f(c_{n})\right\vert <\left( 2m^{2}+1\right) ^{1/p}\left\vert
f(d_{n})-g_{n}(d_{n})\right\vert \text{.}
\end{equation}

As $c_{n},d_{n}\in \left( \alpha _{j}(n),\alpha _{j+1}(n)\right) $, and $%
\lim\limits_{n\rightarrow \infty }\alpha _{j+1}(n)\leq a_{1}$ and $%
f(a_{1})=f(a)=0$, by the monotonicity of $f$, it easily results $%
f(c_{n})\rightarrow 0$ and $f(d_{n})\rightarrow 0$. Then, by the
construction of $g_{n}$, we have $0\leq g_{n}(x)\leq f_{n}(\alpha _{j+1}(n))$%
, for all $x\in \left[ \alpha _{j}(n),\alpha _{j+1}(n)\right] $. In
particular we have $0\leq g_{n}(d_{n})\leq f_{n}(\alpha _{j+1}(n))$. Hence,
since $\lim\limits_{n\rightarrow \infty }f_{n}(\alpha _{j+1}(n))=0$, it
follows that $\lim\limits_{n\rightarrow \infty }$ $g_{n}(d_{n})=0$. This
means that both expressions in the inequality (\ref{ineq2}) converge to $0$.

Next, since $0\leq t_{n}<1$, it follows that $t_{n}\cdot f_{n}\left( \alpha
_{j+1}(n)\right) \rightarrow 0$. On the other hand, as $1-t_{n}\geq \frac{1}{%
2m^{2}+1}$, we have that $f_{n}(a)\rightarrow 0.$ In fact, if $%
(1-t_{n})\cdot f_{n}\left( a\right) +t_{n}\cdot f_{n}\left( \alpha
_{j+1}(n)\right) -f(c_{n})$ would not converge to $0$ we have a
contradiction, since the absolute value of the expression in (\ref{ineq2})
is bounded by an expression converging to $0$. In conclusion, we just proved
that any subsequence of $\left( f_{n}\left( a\right) \right) _{n\geq 1}$
contains a subsequence that converges to $0$ (please note again that all
subsequences were denoted the same to avoid too complicated notations).
Clearly this implies that $\lim\limits_{n\rightarrow \infty }f_{n}(a)=0$.

It remains to prove that $\lim\limits_{n\rightarrow \infty }f_{n}(b)=f(b)$.
First, let us notice that that $\lim\limits_{n\rightarrow \infty
}f_{n}(a)=f(a)$ if $f$ is continuous and nonincreasing with $f_{n}=\mathbf{SD%
}_{m}^{0}(\Delta _{n},f,p)$, for all $n\geq 1$. Indeed, this is immediate
from the first part of the proof taking into account that obviously we have $%
\mathbf{SI}_{m}^{0}(\Delta _{n},-f,p)=-\mathbf{SD}_{m}^{0}(\Delta _{n},f,p)$%
, for all $n\geq 1$. Now, suppose again that $f$ is nondecreasing.
Considering the function $g(x)=f(a+b-x)$, it is easily seen that $\mathbf{SI}%
_{m}^{0}(\Delta _{n},f,p)(b)=\mathbf{SD}_{m}^{0}({\Delta }_{n}^{\prime
},g,p)(a)$, where ${\Delta }_{n}^{\prime }=\left( \beta _{1}(n),\beta
_{2}(n),...,\beta _{k_{n}}(n)\right) $ and $\beta _{i}=a+b-\alpha
_{k_{n}-i+1}$, $i=\overline{1,k_{n}}$. Since $g$ is continuous and $%
\left\Vert {\Delta }_{n}^{\prime }\right\Vert \rightarrow 0$, the previous
result implies $\lim\limits_{n\rightarrow \infty }\mathbf{SD}_{m}^{0}({%
\Delta }_{n}^{\prime },g,p)(a)=g(a)$ which implies that $\lim\limits_{n%
\rightarrow \infty }\mathbf{SI}_{m}^{0}(\Delta _{n},f,p)(b)=f(b)$. It is
clear now that we can apply Lemma \ref{uc nondecreasing}, which means that $%
\lim\limits_{n\rightarrow \infty }\left\Vert \mathbf{SI}_{m}^{0}(\Delta
_{n},f,p)-f\right\Vert =0$, and the proof is complete. $\blacksquare$

\section{Some open questions}

It remains, of course, to extend the results obtained here for $l=0$, to an
arbitrary value of $l$. At this end, let us briefly discuss the case $l\geq
1 $. To do this, let us notice that the crucial point in the proof of the
Main Theorem is that the best $L^{p}$ approximation is constantly negative
until a given knot $\alpha _{j}(n)$ and positive on the following knot $%
\alpha _{j+1}(n)$. In the case $l\geq 1$ this property may not hold, or at
least a different type of reasoning is needed to prove it. Actually, for
important classes of partitions (see the equidistant partitions or even the
partitions based on the \ Chebyshev knots of the first kind, for instance),
it would be suffice to prove that if the function is constantly negative
until a given knot $\alpha _{j}(n)$ then it is positive on the knot $\alpha
_{j+k}(n)$ where $k$ is a constant that does not depend on $n$.

Another interesting problem is to consider instead of the $L^{p}$ norms the
more general approach with monotone norms.

Moreover, it would interesting to find an estimation for the rate of the
uniform convergence of the sequences $\left( \mathbf{SI}_{m}^{0}(\Delta
_{n},f,p)\right) _{n\geq 1}$ and of the sequence $\left( \mathbf{SD}%
_{m}^{0}(\Delta _{n},f,p)\right) _{n\geq 1}$, respectively.

Then of course, as we already mentioned this problem earlier, it would be
interesting to study whether the assumption $m\geq 2l+1$ can be relaxed in
the the statements of Proposition 3 and Corollary 5 by only assuming that $%
m\geq l+1$. Of course, in this case another approach is needed that does not
use Theorem 1 in \cite{Passow} which necessarily implies the limitation $%
m\geq 2l+1$.

Finally, it remains the question whether the main theorem remains true if we
drop the assumptions on the monotonicity of $f$ and of the best piecewise
polynomial approximations. But again, these limitations are necessarily
implied by Theorem 1 in \cite{Shisha} and Theorem 1 in \cite{Passow}.
Therefore, in order to generalize the main theorem for arbitrary not
necessarily monotone functions, a new approach is needed that does not use
the aforementioned theorems.

\newpage

\textbf{Acknowledgment}

The contribution of Lucian Coroianu was possible with the support of
Ministry of Research and Innovation, CNCS-UEFISCDI, project number
PN-III-P1-1.1-PD-2016-1416, within PNCDI III.

\end{document}